\documentclass[11pt,a4paper]{article}
\usepackage{hyperref}
\usepackage[textformat=period]{caption}
\usepackage[english]{babel}
\usepackage[utf8]{inputenc}

\usepackage{amsmath,amsthm,amssymb,amsfonts}
\usepackage{enumerate}
\usepackage{faktor}
\usepackage{dsfont}
\usepackage{textcomp}
\usepackage{graphicx}
\usepackage{extarrows}
\usepackage{epigraph}
\usepackage{graphicx}
\usepackage{subfigure}
\usepackage{sidecap}
\usepackage{indentfirst}
\usepackage{tikz}
\usepackage{pgfplots}
\usetikzlibrary{arrows,matrix,patterns,decorations.markings,positioning}
\usepackage{caption}
\usepackage{float}
\usepackage{color}
\usepackage{epstopdf}
\usepackage{epsfig}
\usepackage{stmaryrd}
\usepackage{color}

\newcommand{\R}{\mathbb{R}}

\newcommand{\Z}{\mathbb{Z}}

\newtheorem{teo}{Theorem}[section]
\newtheorem*{teo*}{Theorem}
\newtheorem{lemma}[teo]{Lemma}
\newtheorem{prop}[teo]{Proposition}
\newtheorem*{prop*}{Proposition}

\newtheorem{cor}[teo]{Corollary}

\begin{document}
\title{On loose Legendrian knots in rational homology spheres}
\author{\scshape{Alberto Cavallo}\\ \\
 \footnotesize{Department of Mathematics and its Application, Central European University,}\\
 \footnotesize{Budapest 1051, Hungary}\\ \\ \small{cavallo\_alberto@phd.ceu.edu}}
\date{}

\maketitle

\begin{abstract}
 We prove that loose Legendrian knots in a rational homology contact 3-sphere, satisfying some additional hypothesis,
 are Legendrian isotopic if and 
 only if they have the same classical invariants. 
 The proof requires a result of Dymara on loose Legendrian knots and Eliashberg's classification of overtwisted contact
 structures on 3-manifolds.
\end{abstract}
\section{Introduction}
Knot theory in contact 3-manifolds turned out to be a very interesting field to study. In this setting, an oriented
knot is called Legendrian if it is everywhere tangent to the contact structure and two such knots
are said to be equivalent if they are Legendrian isotopic; that is, there is an isotopy between them such that they 
are Legendrian at any step. In the last twenty years much work has been done in order to find some criteria to determine
whether two Legendrian knots are Legendrian isotopic or not.
Three invariants can be immediately defined from the definition of Legendrian knot. For this reason
they are usually called classical invariants. 

The first one is the knot type, that is the smooth isotopy class of our oriented Legendrian
knot $K$. The knot type of $K$ is a Legendrian invariant; in fact it is known that two Legendrian 
knots are Legendrian isotopic if and only if there is an ambient contact isotopy of the 3-manifold sending the first knot into
the second one as shown in \cite{Etnyre4}. 

The other two classical invariants are the Thurston-Bennequin number, which
is defined as the linking number of the contact framing of $K$ respect to a Seifert framing of $K$,
and the rotation number; the latter
being the numerical obstruction to extending a non-zero vector field, everywhere 
tangent to the knot, to a Seifert surface of $K$ (see \cite{Etnyre4}). These two
invariants are usually well-defined only for null-homologous knots in a rational homology 3-sphere, but a generalization 
exists for every Legendrian knot in such a manifold. See \cite{Etnyre2} for details.

Legendrian knots in overtwisted contact 3-manifolds come in two types: loose and non-loose. A Legendrian knot is loose if
also its complement is overtwisted, while it is non-loose if the complement is tight. More explicitely, a Legendrian knot is 
loose if and only if we can find an overtwisted disk that is disjoint from the knot.
While it was known that non-loose Legendrian knots are not classified by their classical invariants \cite{LOSS},
in the case of loose knots such example was found only recently by Vogel \cite{Vogel}.
Conversely, there were some results that go in the opposite direction.

Etnyre's coarse classification of loose Legendrian knots \cite{Etnyre5} is probably the most important one. It says
that loose knots are completely determined by their classical invariants, but only up to contactomorphism, which is a 
weaker relation than Legendrian isotopy. Another result was proved by Dymara in \cite{Dymara} and it states
that two Legendrian knots, with same classical invariants, such that the complement of
their union contains an overtwisted disk are 
Legendrian isotopic (Theorem \ref{teo:dymara}). This result holds only in rational homology spheres.

In this paper we show that Dymara's result can be strengthened. In fact we prove the following theorem.
\begin{teo}
 \label{teo:main}
 Consider a rational homology contact 3-sphere $(M,\xi)$. Suppose that there are two loose Legendrian knots $K_1$ and $K_2$ 
 in $(M,\xi)$ such 
 that there exists a pair of disjoint overtwisted disks $(E_1,E_2)$, where $E_i$ is contained in the complement of $K_i$ for 
 $i=1,2$. Then $K_1$ and $K_2$ are Legendrian isotopic if and only if they have the same classical invariants.
\end{teo}
Though we still need an assumption on the overtwisted disks, this version can be applied in many interesting cases like 
disjoint unions of Legendrian knots.
We say that a Legendrian 2-component link $L$ is split if $(M,\xi)$ can be decomposed into $(M_1\#M_2,\xi_1\#\xi_2)$ and 
$K_i\hookrightarrow(M_i,\xi_i)$ for $i=1,2$. In other words, if $L$ is the disjoint union of $K_1$ and $K_2$.
\begin{cor}
 \label{cor:disjoint}
 Suppose $K_1$ and $K_2$ are two loose Legendrian knots in the rational homology contact sphere $(M,\xi)$ such that 
 $K_1\cup K_2$ is a split Legendrian link. 
 Then they are Legendrian isotopic if and only if they have the same classical invariants.
\end{cor}
This paper is organized as follows. In Section \ref{section:main}
we define connected sums for contact 3-manifolds and Legendrian links and
we prove Theorem \ref{teo:main}. In Section \ref{section:disjoint} we explain what 
is the disjoint union of Legendrian knots
and give a precise definition of split Legendrian links. Moreover, we apply our main result to this kind of loose knots.

\paragraph*{Acknowledgements:}
The author would like to thank Andr\'as Stipsicz for suggesting to think about this problem. 
The author is supported by the ERC Grant LDTBud from the Alfr\'ed R\'enyi Institute of Mathematics and a Full Tuition Waiver
for a Doctoral program at Central European University.

\section{A classification theorem for loose Legendrian knots}
\label{section:main}
\subsection{Contact and Legendrian connected sum}
The definition of the connected sum of two 3-manifolds can be easily given also in the contact setting. Let us take two
connected contact manifolds $(M_1,\xi_1)$ and $(M_2,\xi_2)$; we call $M_i'$ (for $i=1,2$) the manifolds obtained from $M_i$ by
removing an open Darboux ball, that is a 3-ball with the standard contact structure, and we define $(M_1\#M_2,\xi_1\#\xi_2)$
the contact
manifold which is gotten by gluing together $M_1'$ and $M_2'$.
The structure $\xi_1\#\xi_2$ is well-defined because it is always possible to glue two contact structures on the boundary
of a Darboux ball. Moreover, the result is independent of the choice of the balls themselves
and we have the following proposition.
\begin{prop}
 \label{prop:disks}
 For every Darboux ball $B$ in the overtwisted manifold $(M,\xi)$ there exists at least one overtwisted disk disjoint from
 $B$.
 In particular, if $(M,\xi)=(M_1\#M_2,\xi_1\#\xi_2)$ and a summand $M_i$ is overtwisted then we
 can always find overtwisted disks which are contained entirely in $M_i'\subset M_1\#M_2$.
\end{prop}
Clearly if $(M_1,\xi_1)$ is overtwisted then the connected sum is still overtwisted, but if both the summands are tight
then it is important to cite the following result of Colin \cite{Colin}.
\begin{teo}
 \label{teo:colin}
 The connected sum of two contact 3-manifolds 
 $(M_1\#M_2,\xi_1\#\xi_2)$ is tight if and only if $(M_1,\xi_1)$ and $(M_2,\xi_2)$ are both tight.
\end{teo}
If a sphere $S$ separates a contact manifold $(M,\xi)$ into two components $M_1$ and $M_2$ such that smoothly it is
$M=M_1\#M_2$ then we can ask whether $\xi$ also splits accordingly. We need to recall the definition of the dividing set 
of a convex surface in a contact 3-manifold. If $v$ is a contact vector field in $(M,\xi)$ transverse to $S$ then
the set $$\Gamma_S=\{x\in S\:|\:v(x)\in\xi_x\}$$
is a collection of curves and arcs in $S$ and is called the dividing set of $S$ in $(M,\xi)$. 
More details can be found in \cite{Etnyre}. Now we can state the following lemma.
\begin{lemma}
 A convex separating 
 sphere $S$ in $(M,\xi)$ gives a connected sum decomposition $(M,\xi)=(M_1\#_S M_2,\xi_1\#_S \xi_2)$ if and only if
 the dividing set $\Gamma_S$ is trivial, which means that consists of a single closed curve.
\end{lemma}
\begin{proof}
 The claim follows from the fact that a contact manifold, whose boundary is a convex sphere, can 
 be glued together with a Darboux ball if and only if its dividing set is trivial as shown in \cite{Etnyre}.
\end{proof}
Etnyre and Honda in \cite{Etnyre3} extended the definition of connected sum to Legendrian links. In order to describe the
construction we need to recall that Legendrian links in the standard contact 3-sphere can be represented 
with their front projection. This is the map  
$$\begin{aligned}
   S^1&\longrightarrow\R^2\\
   \theta\longmapsto&(x(\theta),y(\theta))
   \end{aligned}$$ 
where $(x(\theta),y(\theta),z(\theta))$ is the parametrization of $L$.
The front projection of a Legendrian link $L$ is a special diagram for $L$ with no vertical tangencies, that are replaced by 
cusps, and 
at each crossing, the slope of the overcrossing is smaller than
the one of the undercrossing.

Suppose now that we have two Legendrian links $L_1$ and $L_2$ in $(M_1,\xi_1)$ and $(M_2,\xi_2)$ respectively.
\begin{figure}[ht]
  \centering
  \def\svgwidth{5cm}
  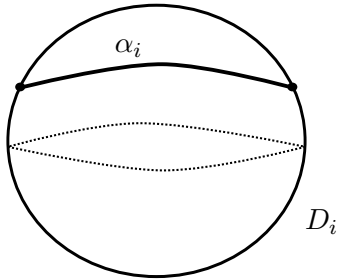   
  \caption{Front projection of $\alpha_i$ in a Darboux ball} 
  \label{Tangle}
\end{figure}
We take two
Darboux balls $D_1$ and $D_2$ as before, but with the condition that $D_i\cap L_i$ is a Legendrian arc $\alpha_i$ where
$\alpha_i\cap\partial D_i$ consists of two points and its front projection is isotopic to the one in Figure \ref{Tangle};
we call this the standard Legendrian tangle.
In this way, the link $L_1\#L_2$ in $(M_1\#M_2,\xi_1\#\xi_2)$ is Legendrian and it does not depend of the choice of the
Darboux balls $D_i$, but only of which component of $L_i$ contains the arc $\alpha_i$.

Let us consider the case of Legendrian links in overtwisted structures. We observe that if $L_1$ is loose then $L_1\#L_2$ is 
also loose, but the connected sum of two non-loose Legendrian links can be loose. 
In fact from \cite{Eliashberg2} we have that in $(S^3,\xi_{-1})$, where
$\xi_i$ is the overtwisted structure on $S^3$ with Hopf invariant equals to $i\in\Z$, 
there is a non-loose Legendrian unknot $K$. 
Then the knot $K\#K$ is a loose Legendrian unknot in $(S^3,\xi_{-2})$, because in \cite{Eliashberg2} it is
also proved that non-loose unknots in $S^3$ only appear in the structure $\xi_{-1}$.

\subsection{Proof of the main result}
In this subsection we prove Theorem \ref{teo:main}.
We only have to show that loose Legendrian knots with same classical invariants are Legendrian isotopic.
We need the following result, proved by Dymara in \cite{Dymara}.
\begin{teo}
  \label{teo:dymara}
  Consider a rational homology overtwisted 3-sphere $(M,\xi)$. Then two loose Legendrian knots $K_1$ and $K_2$ in $(M,\xi)$
  with the same classical invariants and 
  such that $K_1\cup K_2$ is loose are Legendrian isotopic.
\end{teo}
\begin{proof}[Proof of Theorem \ref{teo:main}]
 The idea of the proof is to find another Legendrian knot $K$, with the same classical invariants as $K_1$ and $K_2$, and to 
 show that there are overtwisted disks in the complement of both $K_1\cup K$ and $K_2\cup K$. Then Theorem \ref{teo:dymara}
 gives that $K_1$ and $K_2$ are both Legendrian isotopic to $K$.

 If there is an overtwisted disk in the complement of $K_1\cup K_2$ then the claim follows immediately from Theorem 
 \ref{teo:dymara}. Then we suppose that this is not the case.
 
 Since $E_1$ and $E_2$ are disjoint disks we can find two closed balls $B_1$ and $B_2$ such that $E_i\subset B_i$ for $i=1,2$.
 Moreover, we can suppose that each $B_i$ is disjoint from $K_i$; this is because we start from the assumption that each $E_i$
 is in the complement of $K_i$.
 
 Now we have that the contact manifold $\left(B_i,\xi\lvert_{B_i}\right)$ is an overtwisted $D^3$. Then from Eliashberg's
 classification of overtwisted structures \cite{Eliashberg,Eliashberg3}
 we know that $(B_i,\xi\lvert_{B_i})\#(S^3,\xi_0)$ is contact 
 isotopic to $(B_i,\xi\lvert_{B_i})$. 
 \begin{figure}[ht]
  \centering
  \def\svgwidth{9cm}
  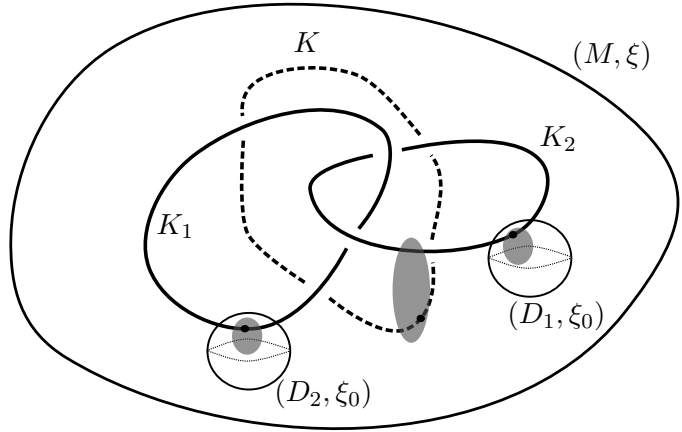   
  \caption{Three overtwisted disks are drawn in grey} 
  \label{Proof}
 \end{figure}
 This holds since the contact connected sum with $(S^3,\xi_0)$ does not change the isotopy class of the 
 2-plane field given by $\xi$. From this we have that inside $(B_i,\xi\lvert_{B_i})$ there is a sphere with trivial 
 dividing set which bounds a closed ball $D_i$ such that $\xi\lvert_{D_i}$ is contact isotopic to $\xi_0$. Moreover, 
 both $\mathring{D_i}$ and $B_i\setminus D_i$ contains some overtwisted disks by Proposition \ref{prop:disks}.
 
 At this point, we have found two closed balls $D_1$ and $D_2$ in $(M,\xi)$ such that each $D_i$ is in the complement 
 of $K_i$ and $\xi\lvert_{D_i}$ is contact isotopic to the overtwisted structure $\xi_0$; furthermore, the contact manifold 
 $M\setminus(D_1\cup D_2)$ is also overtwisted. This situation is pictured in Figure \ref{Proof}. Let us denote with $M'$ the 
 overtwisted manifold $M\setminus D_2$; the boundary of $D_1$ gives a contact connected sum decomposition of $M'$
 where the components are $M\setminus(D_1\cup D_2)$ and $D_1$. 
 Then the same argument that we applied before tells us
 that $(M\setminus(D_1\cup D_2),\xi)$ is contact isotopic to $(M',\xi)$ without a Darboux ball.
 
 Since $K_2$ is a Legendrian knot in $M'$ and we can suppose that a Darboux ball is missed by it, we can identify $K_2$ with a 
 Legendrian knot $K$ inside $M\setminus(D_1\cup D_2)$ with the same classical invariants as $K_2$, which are assumed
 to coincide with the ones of $K_1$. The Legendrian knot $K$ is disjoint from both the overtwisted balls $D_i$ and then 
 we find overtwisted disks in the complement of both $K_i\cup K$. This concludes the proof. 
\end{proof}
The condition on the overtwisted disks in Theorem \ref{teo:main} cannot be removed. Vogel in \cite{Vogel} gives an example
of two loose Legendrian unknots, both with Thurston-Bennequin number equal to zero and rotation number equal to one, that
are not Legendrian isotopic. In fact, although we can find overtwisted disks in the complement of both knots such disks
always intersect themselves.

\section{Disjoint union of Legendrian knots}
\label{section:disjoint}
An important case where the situation described in Theorem \ref{teo:main} appears is when we have a disjoint union of two 
Legendrian knots. In order to explain this application 
we need to define what is a disjoint union in the Legendrian setting, and to show
that is strictly related to the concept of split Legendrian links.

Take $L_i$ Legendrian links in $(M_i,\xi_i)$ for $i=1,2$. We define the disjoint union 
$L_1\sqcup L_2$ in $(M_1\#M_2,\xi_1\#\xi_2)$ as the Legendrian link given by a particular
connected sum $L_1\#\mathcal O_2\#L_2$; where 
$\mathcal O_2$ is the standard Legendrian unlink with 2 components in $(S^3,\xi_{\text{st}})$. 
The connected sum is such that one component of $\mathcal O_2$ is summed to $L_1$ and the other one to $L_2$ as shown in
Figure \ref{disjoint}.
\begin{figure}[ht]
  \centering
  \def\svgwidth{12cm}
  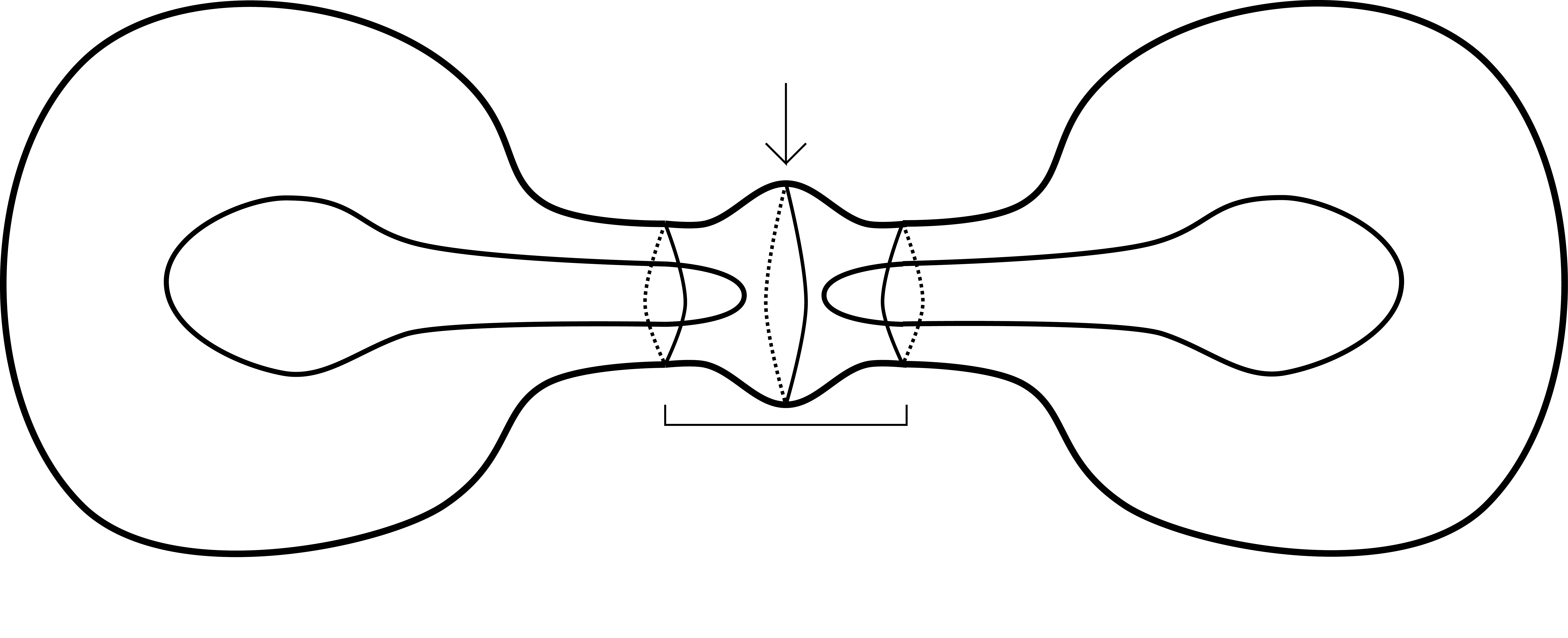   
  \caption{} 
  \label{disjoint}
\end{figure}
Note that $L_1$ and $L_2$ are Legendrian links in $(M_1\#M_2,\xi_1\#\xi_2)$.

In the same way, we say that a Legendrian link $L$, in a contact 3-manifold $(M,\xi)$, is split if 
there exist $L_i\hookrightarrow(M_i,\xi_i)$ for $i=1,2$ such that $(M,\xi)=(M_1,\xi_1)\#(M_2,\xi_2)$ and 
$L$ is Legendrian isotopic to the disjoint union of $L_1$ and $L_2$.
Otherwise, we say that $L$ is non-split.
From this definition we have the following proposition.
\begin{prop}
 \label{prop:disjoint}
 A Legendrian link $L$ in $(M,\xi)$ is split if and only if there are two Legendrian links $L_i\hookrightarrow(M,\xi)$
 and a separating, convex sphere $\Sigma\hookrightarrow(M,\xi)$, whose dividing set is trivial, such that
 \begin{itemize}
  \item The surface $\Sigma$ determines a contact connected sum decomposition of $(M,\xi)$; where the summands are
        $(M_1,\xi_1)$ and $(M_2,\xi_2)$.
  \item Each $L_i$ is embedded in $(M_i,\xi_i)$.
  \item The Legendrian link $L$ is the union of $L_1$ and $L_2$.
 \end{itemize}
\end{prop}
\begin{proof}
 The only if implication is trivial. For the other one we take $B_1$ to be a
 Darboux ball in $(M_1,\xi_1)$ such that $(B_1,B_1\cap L_1)$ is the standard Legendrian tangle as in Figure \ref{Tangle}.
 
 In $B_1$ we find a standard Legendrian unknot $K$, disjoint from $L_1$. We denote by $L_1'$ the new Legendrian link
 obtained from $L_1$ by adding the unlinked component $K$.
 Since $(B_1,B_1\cap L_1)$ is the standard Legendrian tangle, we can find another convex sphere in $B_1$ which gives that
 $L_1'=L_1\#\mathcal O_2$.
 
 At this point, it is clear that $L_1\sqcup L_2=L_1'\#L_2$, where this connected sum is performed between $K$ and
 $L_2$, is a Legendrian link wich is Legendrian isotopic to $L$.
\end{proof}
From Theorem \ref{teo:colin} we observe that if $\xi$ is tight then $L$ is a split Legendrian link if and only if
its smooth link type is split.

\begin{proof}[Proof of Corollary \ref{cor:disjoint}]
 We have that $(M,\xi)=(M_1\#M_2,\xi_1\#\xi_2)$ and $K_i\hookrightarrow(M_i,\xi_i)$. Hence, if $K_1$ is loose in
 $(M_1,\xi_1)$ then clearly we can apply Theorem \ref{teo:dymara}; on the other hand, if $(M_1,\xi_1)$ is tight then
 Theorem \ref{teo:colin} gives that $K_2$ is loose in $(M_2,\xi_2)$ and then we can use the same argument.
 We only have to consider the case when both $K_i$ are non-loose in $(M_i,\xi_i)$. Then we
 apply Theorem \ref{teo:main}, since we can find overtwisted disks in both summands using Proposition \ref{prop:disks}.
\end{proof}
We conclude with the following observation. Let us consider two loose Legendrian knots $K_1$ and $K_2$ in $(M,\xi)$ with
same classical invariants and such that $L=K_1\cup K_2$ is a topologically split 2-component Legendrian link.
Corollary \ref{cor:disjoint} says that if $K_1$ is not Legendrian isotopic to $K_2$ then $L$ is not split as Legendrian link.
The example of Vogel that was mentioned at the end of Section \ref{section:main} falls into this case.


\begin{thebibliography}{9}  
 \bibitem{Etnyre2} K. Baker and J. Etnyre, \emph{Rational linking and contact geometry},
             Progr. Math., 296, Birkh\"auser/Springer, New York, 2012. 
 \bibitem{Colin} V. Colin, \emph{Chirurgies d'indice un et isotopies de sph\`eres dans les vari\'et\'es de contact tendues},
             C. R. Acad. Sci. Paris S\'er. I Math., \textbf{324} (1997), pp. 659-663.
 \bibitem{Dymara} K. Dymara, \emph{Legendrian knots in overtwisted contact structures on $S^3$},
             Ann. Global Anal. Geom., \textbf{19} (2001), no. 3, pp. 293-305. 
 \bibitem{Eliashberg} Y. Eliashberg, \emph{Classification of overtwisted contact structures on 3-manifolds},
             Invent. Math., \textbf{98} (1989), no. 3, pp. 623-637.
 \bibitem{Eliashberg3} Y. Eliashberg, \emph{Classification of contact structures on $\R^3$},
             Internat. Math. Res. Notices., (1993), no. 3, pp. 87-91.
 \bibitem{Eliashberg2} Y. Eliashberg and M. Fraser, \emph{Classification of topologically trivial Legendrian knots},
             CRM Proc. Lecture Notes, 15, Amer. Math. Soc., Providence, RI, 1998. 
 \bibitem{Etnyre4} J. Etnyre, \emph{Legendrian and transversal knots. Handbook of knot theory},
             Elsevier B. V., Amsterdam, 2005, pp. 105-185.            
 \bibitem{Etnyre} J. Etnyre, \emph{Lectures on open book decompositions and contact structures},
             Clay Math., Proc. (5), Amer. Math. Soc., Providence, RI, 2006, pp. 103-141.   
 \bibitem{Etnyre5} J. Etnyre, \emph{On knots in overtwisted contact structures},
             Quantum Topol., \textbf{4} (2013), no. 3, pp. 229-264.            
 \bibitem{Etnyre3} J. Etnyre and K. Honda, \emph{On connected sums and Legendrian knots},
             Adv. Math., \textbf{179} (2003), no. 1, pp. 59-74.  
 \bibitem{LOSS} P. Lisca, P. Ozsv\'ath, A. Stipsicz and Z. Szab\'o, \emph{Heegaard Floer invariants of Legendrian knots in
             contact three–manifolds},
             J. Eur. Math. Soc., \textbf{11} (2009), no. 6, pp. 1307-1363.   
 \bibitem{Vogel} T. Vogel, \emph{Non-loose unknots, overtwisted discs, and the contact mapping class group of $S^3$}, 
             arXiv:1612.06557.
\end{thebibliography}
\end{document}